\documentclass[12pt,reqno]{amsart}
\usepackage{amsmath, amsfonts, amssymb, amsthm}
\def\R{\mathbb R}

\def\1{\mathbf{1}}

\def\P{\mathcal{P}}

\def\vol{\mathrm{vol}}
\def\pmod #1{\ ({\rm mod}\ #1)}

\theoremstyle{plain}
\newtheorem{theorem}{Theorem}
\newtheorem{lemma}{Lemma}

\theoremstyle{definition}

\theoremstyle{remark}
\newtheorem*{remark}{Remark}

\begin{document}
\title{Bounded gaps between primes of a special form}
\author{Hongze Li}
\address{Department of Mathematics, Shanghai Jiao Tong University, Shanghai 200240, People's Public of China}
\email{lihz@sjtu.edu.cn}
\author{Hao Pan}
\address{Department of Mathematics, Nanjing University, Nanjing 210093,
People's Republic of China}
\email{haopan1979@gmail.com}
\begin{abstract} For each $m\geq 1$, there exist infinitely many primes
$
p_1<p_2<\ldots<p_{m+1}
$
such that
$$
p_{m+1}-p_1=O(m^4e^{8m})
$$
and
$p_j+2
$
has at most
$$
\frac{16m}{\log 2}+\frac{5\log m}{\log 2}+37
$$
prime divisors for each $j$.
\end{abstract}
\maketitle

\section{Introduction}
\setcounter{lemma}{0}\setcounter{theorem}{0}\setcounter{corollary}{0}
\setcounter{equation}{0}

A recent breakthrough in prime number theory concerns bounded prime
gaps. In \cite{Z}, with the help of a refined GPY sieve method
\cite{GPY} and an enhanced Bombieri-Vinogradov theorem, Zhang proved
that
\begin{equation}
\liminf_{n\to\infty}(p_{n+1}-p_n)\leq 7\times 10^7,
\end{equation}
where $p_n$ denotes the $n$-th prime. Subsequently, the bound $7\times 10^7$ has been rapidly reduced (cf. \cite{P}). In \cite{M2}, using a multi-dimensional sieve method, Maynard improved the upper bound $600$. Furthermore, using the new sieve method, Maynard and Tao independently proved that
\begin{equation}
\liminf_{n\to\infty}(p_{n+m}-p_n)\leq Cm^3e^{4m}
\end{equation}
for any $m\geq 1$, where $C$ is an absolutely constant.

In fact, using the discussions of Maynard and Tao, one can get a bounded-gaps type result for any subsequence of primes which satisfies the Bombieri-Vinogradov type mean value theorem. Let
$$
\P_d^{(2)}=\{p:\,p\text{ is prime and }\Omega(p+2)\leq d\},
$$
where $\Omega(n)$ denotes the number of prime divisors of $n$.
No asymptotic formula is known for the number of the primes in $\P_d^{(2)}$ less than $x$,
though the well-known Chen theorem asserts that
$$
|\P_2^{(2)}\cap[1,x]|\gg\frac{x}{(\log x)^2}
$$
provided $x$ is sufficiently large.

In this paper, we shall extend the Maynard-Tao theorem to the primes in $\P_d^{(2)}$. Our main result is
\begin{theorem}\label{mt}
Suppose that $m\geq 1$. Then there exist infinitely many primes
$
p_1<p_2<\ldots<p_{m+1}
$
such that
$$
p_{m+1}-p_1=O(m^4e^{8m})
$$
and
\begin{equation}\label{omega}
\Omega(p_j+2)\leq\frac{16m}{\log 2}+\frac{5\log m}{\log 2}+37.
\end{equation}
for each $1\leq j\leq m+1$. In particular, there exist infinitely
many distinct $$p_1,p_2\in\P_{59}^{(2)}$$ such that
$$
|p_2-p_1|\leq 3\times 10^{9}.
$$
\end{theorem}
Throughout this paper, let $\mu$ denote the M\"obius function, $\phi$ denote the Euler 
totient function and let $\tau(n)=\sum_{d\mid n}1$ be the divisor function. And unless indicated otherwise, the constants implied by $\ll$, $\gg$ and $O(\cdot)$ are always absolute.

\section{The Maynard sieve method involving the divisor function}
\setcounter{lemma}{0}\setcounter{theorem}{0}\setcounter{corollary}{0}
\setcounter{equation}{0}

Define the area
$$
\Delta_{k_0,r}=\{(t_1,\ldots,t_{k_0}):\,t_j\geq 0,\ t_1+\cdots+t_{k_0}\leq r\}.
$$
Suppose that $f(t_1,\ldots,t_{k_0})$ is a smooth function supported on $\Delta_{k_0,1}$.
Let
$$
F(t_1,\ldots,t_{k_0})=\frac{\partial^{k_0}f(t_1,\ldots,t_{k_0})}{\partial t_1\cdots\partial t_{k_0}}
$$
and
$$
G_m(t_1,\ldots,t_{k_0})=\frac{\partial F(t_1,\ldots,t_{k_0})}{\partial t_m}.
$$

Suppose that $\{h_1,\ldots,h_{k_0}\}$ is {\it admissible}, i.e., for any prime $p$, $h_1,\ldots,h_{k_0}$ don't cover all residues modulo $p$. Suppose that $x$ is sufficiently large and $x^{c_0}\leq R\leq x^{1/4-c_0}$ for some constant $c_0>0$. Let $w=\log\log\log x$ and
$$
W=\prod_{p\leq w}p.
$$
Since $\{h_1,\ldots,h_{k_0}\}$ is admissible, we may choose $1\leq b\leq W$ such that $(b+h_j,W)=1$ for each $j$.

The following result is motived by the work in \cite{H, HT, M1}.
\begin{theorem}\label{tt} For $1\leq m\leq k_0$,
\begin{align}
&\sum_{\substack{x\leq n<2x\\ n\equiv
b\pmod{W}}}\tau(n+h_m)\bigg(\sum_{d_j\mid n+h_j}f\bigg(\frac{\log
d_1}{\log R},\ldots,\frac{\log d_{k_0}}{\log
R}\bigg)\prod_{j=1}^{k_0}\mu(d_j)\bigg)^2 \notag\\=&\bigg(\frac{\log
x}{\log
R}\cdot\alpha-\beta_1+4\beta_2+o(1)\bigg)\cdot\frac{x}{(\log
R)^{k_0}}\cdot\frac{W^{k_0-1}}{\phi(W)^{k_0}},
\end{align}
where 
\begin{equation}\label{ta}
\alpha=\int_{\Delta_{k_0,1}} t_mG_m(t_1,\ldots,t_{k_0})^2 d t_1\cdots d t_{k_0},
\end{equation}
\begin{equation}
\beta_1=\int_{\Delta_{k_0,1}}t_m^2G_m(t_1,\ldots,t_{k_0})^2d t_1\cdots d t_{k_0}
\end{equation}
and
\begin{equation}
\beta_2=\int_{\Delta_{k_0,1}}t_mG_m(t_1,\ldots,t_{k_0})F(t_1,\ldots,t_{k_0})d t_1\cdots d t_{k_0}.
\end{equation}
\end{theorem}
\begin{proof}
Let
\begin{equation}
\lambda_{d_1,\ldots,d_{k_0}}=f\bigg(\frac{\log d_1}{\log R},\ldots,\frac{\log d_{k_0}}{\log R}\bigg)\prod_{j=1}^{k_0}\mu(d_j).
\end{equation}
Clearly
\begin{align*}
&\sum_{\substack{x\leq n<2x\\ n\equiv b\pmod{W}}}\tau(n+h_m)\bigg(\sum_{d_j\mid n+h_j}\lambda_{d_1,\ldots,d_{k_0}}\bigg)^2\\
=&\sum_{\substack{d_1,\ldots,d_{k_0}\in S_W\\ e_1,\ldots,e_{k_0}\in S_W\\ (d_{j_1}e_{j_1},d_{j_2}e_{j_2})=1}}\lambda_{d_1,\ldots,d_{k_0}}\lambda_{e_1,\ldots,e_{k_0}}
\sum_{\substack{x\leq n<2x\\ n\equiv b\pmod{W}\\ [d_j,e_j]\mid n+h_j}}\tau(n+h_m),
\end{align*}
where
$$
S_W=\{d:\, (d,W)=1\}.
$$
If $q$ is square-free and $r=(a,q)$, then we know (cf. \cite{S}) that
$$
\sum_{\substack{x\leq n<2x\\ n\equiv a\pmod{q}}}\tau(n)=\frac{\tau(r)r}{\phi(r)}\cdot\frac{x\phi(q)}{q^2}\bigg(\log x-\log r+2\gamma-1-2\sum_{p\mid (q/r)}\frac{\log p}{p-1}\bigg)+O_\epsilon(x^{1/3+\epsilon}).
$$
So
\begin{align}\label{tsum}
&\sum_{\substack{x\leq n<2x\\ n\equiv b\pmod{W}}}\tau(n+h_m)\bigg(\sum_{d_j\mid n+h_j}\lambda_{d_1,\ldots,d_{k_0}}\bigg)^2\notag\\
=&\sum_{\substack{d_1,\ldots,d_{k_0}\in S_W\\ e_1,\ldots,e_{k_0}\in S_W\\ (d_{j_1}e_{j_1},d_{j_2}e_{j_2})=1}}\lambda_{d_1,\ldots,d_{k_0}}\lambda_{e_1,\ldots,e_{k_0}}\cdot
\frac{\tau([d_m,e_m])[d_m,e_m]}{\phi([d_m,e_m])}\cdot\frac{\phi(W\prod_{j=1}^{k_0}[d_j,e_j])}{W^2\prod_{j=1}^{k_0}[d_j,e_j]^2}\notag\\
&\cdot x\bigg(\log x-\log([d_m,e_m])+2\gamma-1-2\sum_{p\mid W\prod_{j\neq m}[d_j,e_j]}\frac{\log p}{p-1}\bigg)\notag\\
&+O_\epsilon\bigg(x^{1/3+\epsilon}\sum_{\substack{d_1,\ldots,d_{k_0}\in S_W\\ e_1,\ldots,e_{k_0}\in S_W}}|\lambda_{d_1,\ldots,d_{k_0}}\lambda_{e_1,\ldots,e_{k_0}}|\bigg).
\end{align}

Since $\lambda_{d_1,\ldots,d_{k_0}}=0$ unless $d_1\cdots d_{k_0}\leq R$, it is not difficult to see that the last remainder term in (\ref{tsum})
can be omitted.
Thus we only need to consider
\begin{align*}
M_1=\sum_{\substack{d_1,\ldots,d_{k_0}\\ e_1,\ldots,e_{k_0}\\
(d_{j_1}e_{j_1},d_{j_2}e_{j_2})=1}}\lambda_{d_1,\ldots,d_{k_0}}\lambda_{e_1,\ldots,e_{k_0}}\cdot
\frac{\tau([d_m,e_m])}{[d_m,e_m]}\cdot\frac{\phi(\prod_{j\neq
m}[d_j,e_j])}{\prod_{j\neq m}[d_j,e_j]^2},
\end{align*}
\begin{align*}
M_2=\sum_{\substack{d_1,\ldots,d_{k_0}\\ e_1,\ldots,e_{k_0}\\ (d_{j_1}e_{j_1},d_{j_2}e_{j_2})=1}}\lambda_{d_1,\ldots,d_{k_0}}\lambda_{e_1,\ldots,e_{k_0}}\cdot
\frac{\tau([d_m,e_m])}{[d_m,e_m]}\cdot\frac{\phi(\prod_{j\neq m}[d_j,e_j])}{\prod_{j\neq m}[d_j,e_j]^2}\cdot\log([d_m,e_m]),
\end{align*}
and
\begin{align*}
M_3=\sum_{\substack{d_1,\ldots,d_{k_0}\\ e_1,\ldots,e_{k_0}\\
(d_{j_1}e_{j_1},d_{j_2}e_{j_2})=1}}\lambda_{d_1,\ldots,d_{k_0}}\lambda_{e_1,\ldots,e_{k_0}}\cdot
\frac{\tau([d_m,e_m])}{[d_m,e_m]}\cdot\frac{\phi(\prod_{j\neq
m}[d_j,e_j])}{\prod_{j\neq m}[d_j,e_j]^2}\sum_{p\mid W\prod_{j\neq
m}[d_j,e_j]}\frac{\log p}{p-1}.
\end{align*}

Write
$$
f(t_1,\ldots,t_{k_0})=\int_{\R^{k_0}}\eta(\vec{u})\exp\bigg(-\sum_{j=1}^{k_0}(1+i u_j)t_j\bigg)d\vec{u},
$$
where the vector $\vec{u}=(u_1,\ldots,u_{k_0})$ and
\begin{equation}\label{etau}
\eta(\vec{u})\ll(1+|\vec{u}|)^{-A}
\end{equation}
for any fixed $A>0$.
Then
$$
M_1=\int_{\R^{k_0}}\int_{\R^{k_0}}\eta(\vec{u})\eta(\vec{v})H(\vec{u},\vec{v})d\vec{u}d\vec{v},
$$
where
\begin{align}
H(\vec{s}_1,\vec{s}_2)=&\sum_{\substack{d_1,\ldots,d_{k_0}\in S_W\\ e_1,\ldots,e_{k_0}\in S_W\\ (d_{j_1}e_{j_1},d_{j_2}e_{j_2})=1}}
\frac{\tau([d_m,e_m])}{[d_m,e_m]}\cdot\frac{\phi(\prod_{i\neq m}[d_i,e_i])}{\prod_{j\neq m}[d_j,e_j]^2}\prod_{j=1}^{k_0}\mu(d_j)\mu(e_j)d_j^{-\frac{1+i u_j}{\log R}}e_j^{-\frac{1+i v_j}{\log R}}\notag\\
=&\prod_{p>w}\bigg(1-\sum_{j\neq m}\bigg(\frac{p-1}{p^{2+\frac{1+i u_j}{\log R}}}+\frac{p-1}{p^{2+\frac{1+i v_j}{\log R}}}-\frac{p-1}{p^{2+\frac{1+i u_j}{\log R}+\frac{1+i v_j}{\log R}}}\bigg)\notag\\
&-
\bigg(\frac{2}{p^{1+\frac{1+i u_m}{\log R}}}+\frac{2}{p^{1+\frac{1+i v_m}{\log R}}}-\frac{2}{p^{1+\frac{1+i u_m}{\log R}+\frac{1+i v_m}{\log R}}}\bigg)\bigg).
\end{align}
Since
$$
H(\vec{s}_1,\vec{s}_2)\ll\prod_{p>w}(1+O(p^{-1-1/\log R}))=(\log R)^{O(1)},
$$
in view of (\ref{etau}), we may assume that $|\vec{u}|,|\vec{v}|\leq\sqrt{\log R}$. Then
\begin{align*}
H(\vec{s}_1,\vec{s}_2)=&(1+o(1))\prod_{p>w}\frac{\Big(1-\frac{1}{p^{1+\frac{1+i u_m}{\log R}}}\Big)^2\Big(1-\frac{1}{p^{1+\frac{1+i v_m}{\log R}}}\Big)^2}{\Big(1-\frac{1}{p^{1+\frac{1+i u_m}{\log R}+\frac{1+i v_m}{\log R}}}\Big)^2}\prod_{j\neq m}\frac{\Big(1-\frac{1}{p^{1+\frac{1+i u_j}{\log R}}}\Big)\Big(1-\frac{1}{p^{1+\frac{1+i v_j}{\log R}}}\Big)}{\Big(1-\frac{1}{p^{1+\frac{1+i u_j}{\log R}+\frac{1+i v_j}{\log R}}}\Big)}
\\
=&(1+o(1))\frac{\zeta\big(1+\frac{1+i u_m}{\log R}+\frac{1+i v_m}{\log R}\big)^2\prod_{j\neq m}\zeta\big(1+\frac{1+i u_j}{\log R}+\frac{1+i v_j}{\log R}\big)}{\zeta\big(1+\frac{1+i u_m}{\log R}\big)^2\zeta\big(1+\frac{1+i v_m}{\log R}\big)^2\prod_{j\neq m}\zeta\big(1+\frac{1+i u_j}{\log R}\big)\zeta\big(1+\frac{1+i v_j}{\log R}\big)}\\
&\cdot\prod_{p\leq w}\frac{\Big(1-\frac{1}{p^{1+\frac{1+i u_m}{\log R}+\frac{1+i v_m}{\log R}}}\Big)^2}{\Big(1-\frac{1}{p^{1+\frac{1+i u_m}{\log R}}}\Big)^2\Big(1-\frac{1}{p^{1+\frac{1+i v_m}{\log R}}}\Big)^2}\prod_{j\neq m}\frac{\Big(1-\frac{1}{p^{1+\frac{1+i u_j}{\log R}+\frac{1+i v_j}{\log R}}}\Big)}{\Big(1-\frac{1}{p^{1+\frac{1+i u_j}{\log R}}}\Big)\Big(1-\frac{1}{p^{1+\frac{1+i v_j}{\log R}}}\Big)}
\\
=&(1+o(1))\frac{\big(\frac{1+i u_m}{\log R}\big)^2\big(\frac{1+i v_m}{\log R}\big)^2\prod_{j\neq m}\big(\frac{1+i u_j}{\log R}\big)\big(\frac{1+i v_j}{\log R}\big)}{\big(\frac{1+i u_m}{\log R}+\frac{1+i v_m}{\log R}\big)^2\prod_{j\neq m}\big(\frac{1+i u_j}{\log R}+\frac{1+i v_j}{\log R}\big)}
\cdot\frac{W^{k_0+1}}{\phi(W)^{k_0+1}},
\end{align*}
where $\zeta(s)$ is the Riemann zeta function and we used the fact
$$
\zeta(1+s)^{-1}=s+o(s)
$$
as $s\to 0$.
Hence
$$
M_1=\frac{(1+o(1))W^{k_0+1}}{\Phi(W)^{k_0+1}(\log R)^{k_0+1}}\iint_{\R^{k_0}\times\R^{k_0}}\eta(\vec{u})\eta(\vec{v})L_{1}(\vec{u},\vec{v})d\vec{u}d\vec{v},
$$
where
$$
L_{1}(\vec{u},\vec{v})=\frac{(1+i u_m)^2(1+i v_m)^2}{(1+i u_m+1+i v_m)^2}\prod_{j\neq m}\frac{(1+i u_j)(1+i v_j)}{1+i u_j+1+i v_j}.
$$
Clearly
$$
\frac{\partial^{k+1}f(t_1,\ldots,t_{k_0})}{\partial t_1\cdots\partial^2 t_m\cdots\partial t_{k_0}}=(-1)^{k+1}\int_{\R^{k_0}}\eta(\vec{u})\exp\bigg(-\sum_{j=1}^{k_0}(1+i u_j)t_j\bigg)\cdot(1+i u_m)^2\prod_{j\neq m}(1+i u_j)d\vec{u},
$$
i.e.,
\begin{align*}
G_m(t_1,\ldots,t_{k_0})^2 =
\iint_{\R^{k_0}\times\R^{k_0}}\eta(\vec{u})\eta(\vec{v})K_1(\vec{u},\vec{v})d\vec{u}d\vec{v},
\end{align*}
where
$$
K_1(\vec{u},\vec{v})=\exp\bigg(-\sum_{j=1}^{k_0}(1+i u_j+1+i v_j)t_j\bigg)\cdot(1+i u_m)^2(1+i v_m)^2\prod_{j\neq m}(1+i u_j)(1+i v_j).
$$
Then by Fubini's theorem,
we have
\begin{align*}
&\int_{\Delta_{k_0,\infty}}\bigg(\int_{t_m}^{+\infty}G_m(t_1,\ldots,t_m',\ldots,t_{k_0})^2 d t_m'\bigg)d t_1\cdots d t_{k_0}\\
=&\int_{\Delta_{k_0,\infty}}
\bigg(\iint_{\R^{k_0}\times\R^{k_0}}\frac{K_1(\vec{u},\vec{v})d\vec{u}d\vec{v}}{1+i u_m+1+i v_m}\bigg)d t_1\cdots d t_{k_0}\\
=&\iint_{\R^{k_0}\times\R^{k_0}}\frac{\eta(\vec{u})\eta(\vec{v})K_1(\vec{u},\vec{v})d\vec{u}d\vec{v}}{(1+i u_m+1+i v_m)^2\prod_{j\neq m}(1+i u_j+1+i v_j)}\\
=&\iint_{\R^{k_0}\times\R^{k_0}}\eta(\vec{u})\eta(\vec{v})L_1(\vec{u},\vec{v})d\vec{u}d\vec{v}.
\end{align*}
Finally, since $G_m(t_1,\ldots,t_{k_0})$ is also supported on $\Delta_{k_0,1}$, clearly
\begin{align*}
&\int_{\Delta_{k_0,\infty}}\bigg(\int_{t_m}^{+\infty}G_m(t_1,\ldots,t_m',\ldots,t_{k_0})^2 d t_m'\bigg)d t_1\cdots dt_m\cdots d t_{k_0}\\
=&\int_{\Delta_{k_0,1}} t_{m}'G_m(t_1,\ldots,t_m',\ldots,t_{k_0})^2 d t_1\cdots d t_m'\cdots d t_{k_0}.
\end{align*}

Now let us turn to $M_2$. Similarly, we have
$$
M_2=\iint_{\R^{k_0}\times\R^{k_0}}\eta(\vec{u})\eta(\vec{v})H^*(\vec{u},\vec{v})d\vec{u}d\vec{v},
$$
where
$$
H^*(\vec{s}_1,\vec{s}_2)=\sum_{\substack{d_1,\ldots,d_{k_0}\in S_W\\ e_1,\ldots,e_{k_0}\in S_W\\ (d_{j_1}e_{j_1},d_{j_2}e_{j_2})=1}}\log([d_m,e_m])\cdot
\frac{\tau([d_m,e_m])}{[d_m,e_m]}\prod_{j\neq m}\frac{\phi([d_j,e_j])}{[d_j,e_j]^2}\prod_{j=1}^{k_0}\frac{\mu(d_j)\mu(e_j)}{d_j^{\frac{1+i u_j}{\log R}}e_j^{\frac{1+i v_j}{\log R}}}.
$$
Clearly
\begin{align*}
H^*(\vec{s}_1,\vec{s}_2)
=&\sum_{q>w}\bigg(\frac{2\log q}{q^{1+\frac{1+i u_m}{\log R}+\frac{1+i v_m}{\log R}}}-\frac{2\log q}{q^{1+\frac{1+i u_m}{\log R}}}-\frac{2\log q}{q^{1+\frac{1+i v_m}{\log R}}}\bigg)\\
&\cdot\prod_{\substack{p>w\\ p\neq q}}\bigg(1-\sum_{j\neq m}\bigg(\frac{p-1}{p^{2+\frac{1+i u_j}{\log R}}}+\frac{p-1}{p^{2+\frac{1+i v_j}{\log R}}}-\frac{p-1}{p^{2+\frac{1+i u_j}{\log R}+\frac{1+i v_j}{\log R}}}\bigg)\\
&-
\bigg(\frac{2}{p^{1+\frac{1+i u_m}{\log R}}}+\frac{2}{p^{1+\frac{1+i v_m}{\log R}}}-\frac{2}{p^{1+\frac{1+i u_m}{\log R}+\frac{1+i v_m}{\log R}}}\bigg)\bigg)\\
=&\sum_{q>w}\bigg(\frac{2\log q}{q^{1+\frac{1+i u_m}{\log R}+\frac{1+i v_m}{\log R}}}-\frac{2\log q}{q^{1+\frac{1+i u_m}{\log R}}}-\frac{2\log q}{q^{1+\frac{1+i v_m}{\log R}}}\bigg)\cdot\frac{H(\vec{s}_1,\vec{s}_2)}{1-O(q^{-1})}\\
=&(1+o(1))H(\vec{s}_1,\vec{s}_2)\sum_{q>w}\bigg(\frac{2\log q}{q^{1+\frac{1+i u_m}{\log R}+\frac{1+i v_m}{\log R}}}-\frac{2\log q}{q^{1+\frac{1+i u_m}{\log R}}}-\frac{2\log q}{q^{1+\frac{1+i v_m}{\log R}}}\bigg).
\end{align*}
Note that
$$
\sum_{q>w}\frac{\log q}{q^{1+s}}=\sum_{n=1}^\infty\frac{\Lambda(n)}{n^{1+s}}+O((\log w)^2)=-\frac{\zeta'(1+s)}{\zeta(1+s)}+O((\log w)^2)
$$
and
$$
-\frac{\zeta'(1+s)}{\zeta(1+s)}=s+o(s)
$$
as $s\to 0$.
So
$$
M_2=\frac{(2+o(1))W^{k_0+1}}{\Phi(W)^{k_0+1}(\log R)^{k_0}}\iint_{\R^{k_0}\times\R^{k_0}}\eta(\vec{u})\eta(\vec{v})(L_2(\vec{u},\vec{v})-L_3(\vec{u},\vec{v})-L_3(\vec{v},\vec{u}))d\vec{u}d\vec{v},
$$
where
$$
L_2(\vec{u},\vec{v})=\frac{(1+i u_m)^2(1+i v_m)^2}{(1+i u_m+1+i v_m)^3}\prod_{j\neq m}\frac{(1+i u_j)(1+i v_j)}{1+i u_j+1+i v_j}
$$
and
$$
L_3(\vec{u},\vec{v})=\frac{(1+i u_m)(1+i v_m)^2}{(1+i u_m+1+i v_m)^2}.
\prod_{j\neq m}\frac{(1+i u_j)(1+i v_j)}{1+i u_j+1+i v_j}.
$$

Clearly
\begin{align*}&\frac12\int_{\Delta_{k_0,1}} t_{m}''^2G_m(t_1,\ldots,t_m'',\ldots,t_{k_0})^2 d t_1\cdots d t_m''\cdots d t_{k_0}
\\=
&\int_{\Delta_{k_0,\infty}}\bigg(\int_{t_m}^{+\infty}\bigg(\int_{t_m'}^{+\infty}G_m(t_1,\ldots,t_m'',\ldots,t_{k_0})^2d t_m''\bigg) d t_m'\bigg)d t_1\cdots d t_{k_0}\\
=&\int_{\Delta_{k_0,\infty}}
\bigg(\iint_{\R^{k_0}\times\R^{k_0}}\frac{\eta(\vec{u})\eta(\vec{v})K_1(\vec{u},\vec{v})d\vec{u}d\vec{v}}{(1+i u_m+1+i v_m)^2}\bigg)d t_1\cdots d t_{k_0}\\
=&\iint_{\R^{k_0}\times\R^{k_0}}\eta(\vec{u})\eta(\vec{v})L_2(\vec{u},\vec{v})d\vec{u}d\vec{v}.
\end{align*}
Let
$$
K_2(\vec{u},\vec{v})=\exp\bigg(-\sum_{j=1}^{k_0}(1+i u_j+1+i v_j)t_j\bigg)(1+i u_m)(1+i v_m)^2\prod_{j\neq m}(1+i u_j)(1+i v_j).
$$
Then
$$
G_m(t_1,\ldots,t_{k_0})F(t_1,\ldots,t_{k_0})
=\iint_{\R^{k_0}\times\R^{k_0}}\eta(\vec{u})\eta(\vec{v})K_2(\vec{u},\vec{v})d\vec{u}d\vec{v}.
$$
Similarly, we also have
\begin{align*}&\int_{\Delta_{k_0,1}} \bigg(\int_{t_m}^\infty G_m(t_1,\ldots,t_m',\ldots,t_{k_0})F(t_1,\ldots,t_m',\ldots,t_{k_0})d t_m'\bigg) d t_1\cdots d t_{k_0}
\\
=&\int_{\Delta_{k_0,\infty}}
\bigg(\iint_{\R^{k_0}\times\R^{k_0}}\frac{K_2(\vec{u},\vec{v})d\vec{u}d\vec{v}}{1+i u_m+1+i v_m}\bigg)d t_1\cdots d t_{k_0}=\iint_{\R^{k_0}\times\R^{k_0}}L_3(\vec{u},\vec{v})d\vec{u}d\vec{v}.
\end{align*}

Finally, using the similar disscusions, it is not difficult to see that
$$
M_3=o\bigg(\frac{W^{k_0+1}}{\phi(W)^{k_0+1}}\cdot\frac{1}{(\log R)^{k_0}}\bigg).
$$
All are done. \end{proof}

\section{Proof of Theorem \ref{mt}}
\setcounter{lemma}{0}\setcounter{theorem}{0}\setcounter{corollary}{0}
\setcounter{equation}{0}

\begin{lemma}\label{ad}
For $k_0\geq 1$, there exist $h_1<h_2<\cdots<h_{2k_0}$ such that $h_{2j}=h_{2j-1}+2$ for $1\leq j\leq k_0$,
$$
\{h_1,h_2,\ldots,h_{2k_0-1},h_{2k_0}\}
$$
is admissible and
$$
h_{2k_0}-h_1=O(k_0(\log k_0)^2).
$$
\end{lemma}
\begin{proof}
Let $z=Ck_0^9$ where $C$ is a sufficiently large constant. Let
$$
S=\{q\in[z,2z]:\, \text{all prime divisors of }q(q+2)\text{ are greater than }2k_0\}.
$$
By the Jurkat-Richert theorem (cf. \cite[Theorem 8.4]{HR}), we know that
$$
\frac{z}{(\log z)^2}\ll|S|\ll\frac{z}{(\log z)^2}.
$$
Thus there exists $n\in[z-L,2z]$ such that
$$
|[n,n+L]\cap S|\geq k_0,
$$
where $L=2zk_0/|S|$.
Choose
$$
h_1,h_3,\ldots,h_{2k_0-1}\in [n,n+L]\cap S
$$
and let $h_{2j}=h_{2j-1}+2$. We are done.
\end{proof}

Suppose that $x$ is sufficiently large and let $R=x^{1/4-1/(1000m)}.$
Define
$$
\theta(n)=\begin{cases}\log n,&\text{if }n\text{ is prime},\\ 0,&\text{otherwise.}\end{cases}
$$
Let
\begin{equation}\label{C1}
k_0=m^2e^{8m+8}.
\end{equation}
Suppose that $\{h_1,\ldots,h_{2k_0}\}$ is an admissible set described in Lemma \ref{ad} and $f(t_1,\ldots,t_{2k_0})$ is a smooth function supported on $\Delta_{2k_0,1}$.
We need to show that the sum
\begin{align}\label{C2}
\sum_{\substack{x\leq n<2x\\ n\equiv b\pmod{W}\\ \mu(n+h_{2j})\neq0,\,1\leq j\leq k}}\bigg(\sum_{j=1}^{k_0}\theta(n+h_{2j-1})\bigg(1-\frac{\tau(n+h_{2j})}{C_2}\bigg)-m\log(3x)\bigg)
\bigg(\sum_{d_j\mid n+h_j}\lambda_{d_1,\ldots,d_{2k_0}}\bigg)^2
\end{align}
is positive, where
$$
\lambda_{d_1,\ldots,d_{2k_0}}=f\bigg(\frac{\log d_1}{\log R},\ldots,\frac{\log d_{2k_0}}{\log R}\bigg)\prod_{j=1}^{2k_0}\mu(d_j)
$$
and $C_2$ is a constant depending on $k_0$ to be chosen later. Then
there exist distinct $1\leq j_1,\ldots,j_{m+1}\leq k_0$ such that
$$
\theta(n+h_{2j_i-1})\bigg(1-\frac{\tau(n+h_{2j_i})}{C_2}\bigg)>0,
$$
i.e., $n+h_{2j_i-1}$ is prime and $\tau(n+h_{2j_i})<C_2.$ Since
$\mu(n+h_{2j_i})\neq0$, we get that
$$\Omega(n+h_{2j_i})=\Omega(n+h_{2j_i-1}+2)\leq \frac{\log C_2}{\log
2}.$$

According to \cite[Lemma 4.1]{P2}, we can get
\begin{align}\label{tss}
&\sum_{j=1}^{k_0}\sum_{\substack{x\leq n<2x\\ n\equiv b\pmod{W}}}\theta(n+h_{2j-1})\bigg(\sum_{d_j\mid n+h_j}f\bigg(\frac{\log d_1}{\log R},\ldots,\frac{\log d_{2k_0}}{\log R}\bigg)\prod_{j=1}^{2k_0}\mu(d_j)\bigg)^2\notag\\
=&\frac{(k_0+o(1))x}{(\log R)^{2k_0-1}}\cdot\frac{W^{2k_0-1}}{\phi(W)^{2k_0}}\int_{\Delta_{2k_0-1,1}}\bigg(\int_0^1F(t_1,\ldots,t_{2k_0})d t_{2k_0}\bigg)^2dt_1\cdots d t_{2k_0-1}
\end{align}
and
\begin{align}\label{ss}
&\log(3x)\sum_{\substack{x\leq n<2x\\ n\equiv b\pmod{W}}}\bigg(\sum_{d_j\mid n+h_j}f\bigg(\frac{\log d_1}{\log R},\ldots,\frac{\log d_{2k_0}}{\log R}\bigg)\prod_{j=1}^{2k_0}\mu(d_j)\bigg)^2\notag\\
=&\frac{(1+o(1))x\log x}{(\log R)^{2k_0}}\cdot\frac{W^{2k_0-1}}{\phi(W)^{2k_0}}\int_{\Delta_{2k_0,1}}F(t_1,\ldots,t_{2k_0})^2d t_1\cdots d t_{2k_0},
\end{align}
where
$$
F(t_1,\ldots,t_{2k_0})=\frac{\partial^{2k_0}f(t_1,\ldots,t_{2k_0})}{\partial
t_1\cdots\partial t_{2k_0}}.
$$

Clearly for any $1\leq j_0\leq 2k_0$ and prime $p\in(w,x^{1/2}]$,
\begin{align*}
&\sum_{\substack{x\leq n<2x\\ n\equiv b\pmod{W}\\ n+h_{2j_0}\equiv 0\pmod{p^2}}}\bigg(\sum_{d_j\mid n+h_j}\lambda_{d_1,\ldots,d_{2k_0}}\bigg)^2\\
\leq&2\sum_{\substack{x\leq n<2x\\ n\equiv b\pmod{W}\\ n+h_{2j_0}\equiv 0\pmod{p^2}}}\bigg(\sum_{\substack{d_j\mid n+h_j\\ p\nmid d_{2j_0}}}\lambda_{d_1,\ldots,d_{2k_0}}\bigg)^2+2\sum_{\substack{x\leq n<2x\\ n\equiv b\pmod{W}\\ n+h_{2j_0}\equiv 0\pmod{p^2}}}\bigg(\sum_{\substack{d_j\mid n+h_j\\ p\mid d_{2j_0}}}\lambda_{d_1,\ldots,d_{2k_0}}\bigg)^2.
\end{align*}
According to the Maynard sieve method, we have
\begin{align*}
&\sum_{\substack{x\leq n<2x\\ n\equiv b\pmod{W}\\ n+h_{2j_0}\equiv 0\pmod{p^2}}}\bigg(\sum_{\substack{d_j\mid n+h_j\\ p\nmid d_{2j_0}}}\lambda_{d_1,\ldots,d_{2k_0}}\bigg)^2\\
=&\frac{(1+o(1))x}{(\log R)^{2k_0}}\cdot\frac{(p^2W)^{2k_0-1}}{\phi(p^2W)^{2k_0}}\int_{\Delta_{2k_0},1}F(t_1,\ldots,t_{2k_0})^2 d t_1\cdots d t_{2k_0}+O(x^\epsilon)
\end{align*}
and
\begin{align*}
&\sum_{\substack{x\leq n<2x\\ n\equiv b\pmod{W}\\ n+h_{2j_0}\equiv 0\pmod{p^2}}}\bigg(\sum_{\substack{d_j\mid n+h_j\\ p\mid d_{2j_0}}}\lambda_{d_1,\ldots,d_{2k_0}}\bigg)^2\\
=&\frac{(1+o(1))x}{(\log R)^{2k_0}}\cdot\frac{(p^2W)^{2k_0-1}}{\phi(p^2W)^{2k_0}}\int_{\Delta_{2k_0},1}F\bigg(t_1,\ldots,t_{2j_0}+\frac{\log p}{\log R},\ldots,t_{2k_0}\bigg)^2 d t_1\cdots d t_{2k_0}+O(x^\epsilon).
\end{align*}
Hence
\begin{align*}
&\sum_{j=1}^{k_0}\sum_{\substack{x\leq n<2x\\ n\equiv b\pmod{W}\\ \mu(n+h_{2j'})=0\text{ for some }j'}}\theta(n+h_{2j-1})\bigg(\sum_{d_j\mid n+h_j}\lambda_{d_1,\ldots,d_{2k_0}}\bigg)^2\\
\leq&\log(2x)\sum_{1\leq j,j'\leq k_0}\sum_{p>w}\sum_{\substack{x\leq n<2x\\ n\equiv b\pmod{W}\\ n+h_{2j'}\equiv 0\pmod{p^2}}}\bigg(\sum_{d_j\mid n+h_j}\lambda_{d_1,\ldots,d_{2k_0}}\bigg)^2\\
\ll&\frac{x\log x}{(\log R)^{2k_0}}\cdot\frac{W^{2k_0-1}}{\phi(W)^{2k_0}}\cdot\sum_{p>w}\frac{1}{p^2}.
\end{align*}
Since $w$ tends to infinity as $x\to\infty$, we have
\begin{align}\label{tm}
&\sum_{\substack{x\leq n<2x\\ n\equiv b\pmod{W}\\ \mu(n+h_{2j})\neq0,\,1\leq j\leq k}}\theta(n+h_{2j-1})\bigg(\sum_{d_j\mid n+h_j}\lambda_{d_1,\ldots,d_{2k_0}}\bigg)^2\notag\\
=&(1+o(1))\sum_{\substack{x\leq n<2x\\ n\equiv b\pmod{W}}}\theta(n+h_{2j-1})\bigg(\sum_{d_j\mid n+h_j}\lambda_{d_1,\ldots,d_{2k_0}}\bigg)^2.
\end{align}

Now we shall construct the function $F(t_1,\ldots,t_{2k_0})$ and apply Theorem \ref{tt} to compute
$$
\sum_{\substack{x\leq n<2x\\ n\equiv b\pmod{W}}}\sum_{j=1}^{k_0}\tau(n+h_{2j})
\bigg(\sum_{d_j\mid n+h_j}\lambda_{d_1,\ldots,d_{2k_0}}\bigg)^2.
$$
From now on we only consider $j=k_0$.
Let \begin{equation}\label{d1}\delta_1=\frac{1}{4.5k_0\log k_0}.\end{equation}
Let $h_1(t_1,\ldots,t_{2k_0})$ be a smooth function with $|h_1(t_1,\ldots,t_{2k_0})|\leq 1$ such that
$$
h_1(t_1,\ldots,t_{2k_0})=\begin{cases}1,\qquad&\text{if }(t_1,\ldots,t_{2k_0})\in\Delta_{2k_0,1-\delta_1},\\
0,&\text{if }(t_1,\ldots,t_{2k_0})\not\in\Delta_{2k_0,1}.
\end{cases}
$$
Furthermore, we may assume that
$$
\bigg|\frac{\partial h_1}{\partial t_j}(t_1,\ldots,t_{2k_0})\bigg|\leq\frac{1}{\delta_1}+1
$$
for each $(t_1,\ldots,t_{2k_0})\in \Delta_{2k_0,1}\setminus
\Delta_{2k_0,1-\delta_1}$ and $1\leq j\leq 2k_0$.

Let
$$
A=\log(2k_0)-2\log\log(2k_0)
$$
and $$T=\frac{e^A-1}{A}.$$ It is easy to verify
$$A>0.69\log k_0.$$ Let
$$
\delta_2=\frac{\delta_1 T}{10}.
$$
We also have
\begin{equation}\label{d2}
\delta_2\geq \frac{\log k_0}{2k_0}.
\end{equation}
Let $h_2(t)$ be a smooth function with $|h_2(t)|\leq 1$ such that
$$
h_2(t)=\begin{cases}1,&\text{if }\delta_3\leq t\leq T-\delta_2,\\
0,&\text{if }t>T\text{ or }t<0,
\end{cases}
$$
where $\delta_3>0$ is a small constant to be chosen soon.
Furthermore, we may assume that
$$
|h_2'(t)|\leq\frac{1}{\delta_2}+1
$$
for $t\in[T-\delta_2,T]$ and
$$
|h_2'(t)|\leq\frac{1}{\delta_3}+1
$$
for $t\in[0,\delta_3]$

Let
$$
F(t_1,\ldots,t_{2k_0})=h_1(t_1,\ldots,t_{2k_0})\prod_{j=1}^{2k_0}\frac{h_2(2k_0t_j)}{1+2k_0At_j}
$$
and
$$
\gamma=\frac{1}{A}\bigg(1-\frac1{1+AT}\bigg).
$$
Since
$$
1-\frac{A}{e^A-1}-\frac{e^A}{2k_0}>0,
$$
according to the discussions of
Maynard \cite{M2}, we have
\begin{equation}\label{Fci1}
\int_{\Delta_{2k_0,1}}F^{\circ}(t_1,\ldots,t_{2k_0})^2d t_1\cdots d t_{2k_0}\leq \frac{\gamma^{2k_0}}{(2k_0)^{2k_0}}
\end{equation}
and
\begin{align}\label{Fci2}
&\int_{\Delta_{2k_0-1,1}}\bigg(\int_0^\infty F^{\circ}(t_1,\ldots,t_{2k_0})d t_{2k_0}\bigg)^2d t_1\cdots d t_{2k_0-1}\notag\\
\geq&\frac{\log(2k_0)-2\log\log(2k_0)-2}{2k_0}\cdot \frac{\gamma^{2k_0}}{(2k_0)^{2k_0}},
\end{align}
where
$$
F^{\circ}(t_1,\ldots,t_{2k_0})=\prod_{j=1}^{2k_0}\frac{\1_{[0,T]}(2k_0t_j)}{1+2k_0At_j}.
$$

Define
$$
h_2^*(t)=\begin{cases} 1,&\text{if }0\leq t\leq T-\delta_2,\\
h_2(t),&\text{otherwise},
\end{cases}
$$
and
$$
F^*(t_1,\ldots,t_{2k_0})=h_1(t_1,\ldots,t_{2k_0})\prod_{j=1}^{2k_0}\frac{h_2^*(2k_0t_j)}{1+2k_0At_j}.
$$
Clearly $$F^*(t_1,\ldots,t_{2k_0})=F(t_1,\ldots,t_{2k_0})$$ unless $0\leq t_j\leq \delta_3$ for some $j$.
Now we may choose $\delta_3>0$ sufficiently small such that
\begin{align}\label{fs}&\int_{\Delta_{2k_0-1,1}}\bigg(\int_0^\infty F(t_1,\ldots,t_{2k_0-1},t_{2k_0})d t_{2k_0}\bigg)^2d t_1\cdots d t_{2k_0-1}\notag\\
\geq&
(1-\delta_1)^{0.001}\int_{\Delta_{2k_0-1,1}}\bigg(\int_0^\infty F^*(t_1,\ldots,t_{2k_0-1},t_{2k_0})d t_{2k_0}\bigg)^2d t_1\cdots d t_{2k_0-1}.
\end{align}
Consider
$$
\Delta_{2k_0,r,s}=\{(t_1,\ldots,t_{2k_0},t_{2k_0}'):\,t_1,\ldots,t_{2k_0},t_{2k_0}'\in[0,s],\ t_1+\cdots+t_{2k_0},t_1+\cdots+t_{2k_0}'\leq r\}.
$$
It is easy to see that
$$
\frac{\vol(\Delta_{2k_0,1-\delta_1,(T-\delta_2)/(2k_0)})}{\vol(\Delta_{2k_0,1,T/(2k_0)})}\geq(1-\delta_1)^{2k_0+1}(1-\delta_2T^{-1})^{2k_0+1}\geq
1-2.24k_0\delta_1,
$$
where $\vol(\Delta)$ denotes the volume of $\Delta$.
We have $$F^{*}(t_1,\ldots,t_{2k_0})=F^{\circ}(t_1,\ldots,t_{2k_0})$$
provided $(t_1,\ldots,t_{2k_0})\in \Delta_{2k_0,1-\delta_1}$ and $t_1,\ldots,t_{2k_0}\in[0,(T-\delta_2)/(2k_0)]$.
Also note that $F^{\circ}(t_1,\ldots,t_{2k_0})$ is decreasing in those $t_j$. Hence we have
\begin{align}\label{fj}
&\int_{\Delta_{2k_0-1,1}}\bigg(\int_0^\infty F^*(t_1,\ldots,t_{2k_0-1},t_{2k_0})d t_{2k_0}\bigg)^2d t_1\cdots d t_{2k_0-1}\notag\\
\geq&\int_{\Delta_{2k_0,1-\delta_1,(T-\delta_2)/(2k_0)}}F^\circ(t_1,\ldots,t_{2k_0-1},t_{2k_0})F^\circ(t_1,\ldots,t_{2k_0-1},t_{2k_0}')d t_1\cdots d t_{2k_0}d t_{2k_0}'\notag\\
\geq&(1-2.24k_0\delta_1)\int_{\Delta_{2k_0,1,T/(2k_0)}}F^\circ(t_1,\ldots,t_{2k_0-1},t_{2k_0})F^\circ(t_1,\ldots,t_{2k_0-1},t_{2k_0}')d t_1\cdots d t_{2k_0}d t_{2k_0}'\notag\\
=&(1-2.24k_0\delta_1)\int_{\Delta_{2k_0-1,1}}\bigg(\int_0^1F^\circ(t_1,\ldots,t_{2k_0})d t_{2k_0}\bigg)^2dt_1\cdots d t_{2k_0-1}.
\end{align}
It is easy to check that
$$
1-2.25k_0\delta_1= 1-\frac{2.25k_0}{4.5k_0\log
k_0}\geq\frac{\log(2k_0)-2\log\log(2k_0)-2.5}{\log(2k_0)-2\log\log(2k_0)-2}.
$$
Thus by (\ref{fs}) and (\ref{fj}),
\begin{align*}
&\int_{\Delta_{2k_0-1,1}}\bigg(\int_0^1F(t_1,\ldots,t_{2k_0})d t_{2k_0}\bigg)^2dt_1\cdots d t_{2k_0-1}\\
\geq&(1-2.25k_0\delta_1)\int_{\Delta_{2k_0-1,1}}\bigg(\int_0^1F^\circ(t_1,\ldots,t_{2k_0})d t_{2k_0}\bigg)^2dt_1\cdots d t_{2k_0-1}\\
\geq&\frac{\log(2k_0)-2\log\log(2k_0)-2.5}{2k_0}\cdot\frac{\gamma^{2k_0}}{(2k_0)^{2k_0}}.
\end{align*}
Using (\ref{tss}), (\ref{tm}), (\ref{Fci2}) and noting that
$$
\log(2k_0)-2\log\log(2k_0)\geq 8m+3,
$$
we get
\begin{align}\label{tsse}
&\sum_{j=1}^{k_0}\sum_{\substack{x\leq n<2x\\ n\equiv b\pmod{W}\\ \mu(n+h_{2j})\neq0,\,1\leq j\leq k}}\theta(n+h_{2j-1})\bigg(\sum_{d_j\mid n+h_j}\lambda_{d_1,\ldots,d_{2k_0}}\bigg)^2\notag\\
\geq&\frac{(k_0+o(1))x}{(\log R)^{2k_0-1}}\cdot\frac{W^{2k_0-1}}{\phi(W)^{2k_0}}\cdot\frac{8m+0.5}{2k_0}\cdot\frac{\gamma^{2k_0}}{(2k_0)^{2k_0}}.
\end{align}
Furthermore, in view of (\ref{ss}) and (\ref{Fci1}), we also have
\begin{align}\label{sse}
&\log(3x)\sum_{\substack{x\leq n<2x\\ n\equiv b\pmod{W}}}\bigg(\sum_{d_j\mid n+h_j}f\bigg(\frac{\log d_1}{\log R},\ldots,\frac{\log d_{2k_0}}{\log R}\bigg)\prod_{j=1}^{2k_0}\mu(d_j)\bigg)^2\notag\\
\leq&\frac{(1+o(1))x\log x}{(\log R)^{2k_0}}\cdot\frac{W^{2k_0-1}}{\phi(W)^{2k_0}}\int_{\Delta_{2k_0,1}}F^\circ(t_1,\ldots,t_{2k_0})^2d t_1\cdots d t_{2k_0}\notag\\
\leq&\frac{(1+o(1))x\log x}{(\log R)^{2k_0}}\cdot\frac{W^{2k_0-1}}{\phi(W)^{2k_0}}\cdot
\frac{\gamma^{2k_0}}{(2k_0)^{2k_0}}
\end{align}

Let
\begin{align}
&G_{2k_0}(t_1,\ldots,t_{2k_0})=\frac{\partial F(t_1,\ldots,t_{2k_0})}{\partial t_{2k_0}}\notag\\
=&
\frac{\partial h_1(t_1,\ldots,t_{2k_0})}{\partial t_m}\prod_{j=1}^{2k_0}\frac{h_2(2k_0t_j)}{1+2k_0At_j}
+h_1(t_1,\ldots,t_{2k_0})\cdot\frac{2k_0h_2'(2k_0t_{2k_0})}{1+2k_0At_{2k_0}}\prod_{j=1}^{2k_0-1}\frac{h_2(2k_0t_j)}{1+2k_0At_j}\notag\\
&-
h_1(t_1,\ldots,t_{2k_0})\cdot\frac{2k_0Ah_2(2k_0t_{2k_0})}{(1+2k_0At_{2k_0})^2}\prod_{j=1}^{2k_0-1}\frac{h_2(2k_0t_j)}{1+2k_0At_j}.
\end{align}
By the Cauchy-Schwarz inequality,
\begin{align*}
&G_{2k_0}(t_1,\ldots,t_{2k_0})^2\leq3\bigg(\frac{\partial h_1}{\partial t_m}\prod_{j=1}^{2k_0}\frac{h_2(2k_0t_j)}{1+2k_0At_j}\bigg)^2\\
&
+3\bigg(h_1\cdot\frac{2k_0h_2'(2k_0t_{2k_0})}{1+2k_0At_{2k_0}}\prod_{j=1}^{2k_0-1}\frac{h_2(2k_0t_j)}{1+2k_0At_j}\bigg)^2+3\bigg(h_1\cdot\frac{2k_0Ah_2(2k_0t_{2k_0})}{(1+2k_0At_{2k_0})^2}\prod_{j=1}^{2k_0-1}\frac{h_2(2k_0t_j)}{1+2k_0At_j}\bigg)^2.
\end{align*}
Thus in view of (\ref{ta}),
\begin{align}
&\alpha=\int_{\Delta_{2k_0,1}}t_{2k_0}G_{2k_0}(t_1,\ldots,t_{2k_0-1},t_{2k_0})^2d t_1\cdots d t_{2k_0-1} d t_{2k_0}\notag
\\\leq&
\frac{3+7\delta_1}{\delta_1^2}\int_{\Delta_{2k_0,1}\setminus\Delta_{2k_0,1-\delta_1}}\frac{\1_{[1,T]}(2k_0t_{2k_0})t_{2k_0}}{(1+2k_0At_{2k_0})^2}\prod_{j=1}^{2k_0-1}\frac{\1_{[1,T]}(2k_0t_j)}{(1+2k_0At_j)^2}d
t_1\cdots d t_{2k_0-1}d t_{2k_0}\notag
\\&+\frac{3+7\delta_2}{\delta_2^2}\int_{\Delta_{2k_0-1,1}}\prod_{j=1}^{2k_0-1}\frac{\1_{[1,T]}(2k_0t_j)}{(1+2k_0At_j)^2}\bigg(\int_{(T-\delta_2)/(2k_0)}^{T/(2k_0)}\frac{4k_0^2t_{2k_0}d t_{2k_0}}{(1+2k_0At_{2k_0})^2}\bigg) d t_1\cdots d t_{2k_0-1}\notag
\\&+\frac{3+7\delta_3}{\delta_3^2}\int_{\Delta_{2k_0-1,1}}\prod_{j=1}^{2k_0-1}\frac{\1_{[1,T]}(2k_0t_j)}{(1+2k_0At_j)^2}\bigg(\int_{0}^{\delta_3/(2k_0)}\frac{4k_0^2t_{2k_0}d t_{2k_0}}{(1+2k_0At_{2k_0})^2}\bigg) d t_1\cdots d t_{2k_0-1}\notag\\
&+3\int_{\Delta_{2k_0,1}}\frac{4k_0^2A^2\1_{[1,T]}(2k_0t_{2k_0})t_{2k_0}}{(1+2k_0At_{2k_0})^4}\prod_{j=1}^{2k_0-1}\frac{\1_{[1,T]}(2k_0t_j)}{(1+2k_0At_j)^2}d t_1\cdots d t_{2k_0-1}d t_{2k_0}.
\end{align}
First, clearly
\begin{align}
&\int_{\Delta_{2k_0,1}}\frac{4k_0^2A^2\1_{[1,T]}(2k_0t_{2k_0})t_{2k_0}}{(1+2k_0At_{2k_0})^4}\prod_{j=1}^{2k_0-1}\frac{\1_{[1,T]}(2k_0t_j)}{(1+2k_0At_j)^2}d t_1\cdots d t_{2k_0}\notag\\
\leq&\int_{0}^\infty\frac{4k_0^2A^2t_{2k_0}d t_{2k_0}}{(1+2k_0At_{2k_0})^4}\cdot\bigg(\int_{0}^{T/{2k_0}}\frac{d t}{(1+2k_0At)^2}\bigg) ^{2k_0-1}\leq\frac{4k_0^2A^2}{6\cdot(2k_0A)^2}\cdot\frac{\gamma^{2k_0-1}}{(2k_0)^{2k_0-1}},
\end{align}
Next, we have
\begin{align}
&\int_{\Delta_{2k_0-1,1}}\prod_{j=1}^{2k_0-1}\frac{\1_{[1,T]}(2k_0t_j)}{(1+2k_0At_j)^2}\bigg(\int_{(T-\delta_2)/(2k_0)}^{T/(2k_0)}\frac{4k_0^2t_{2k_0}d t_{2k_0}}{(1+2k_0At_{2k_0})^2}\bigg) d t_1\cdots d t_{2k_0-1}\notag\\
\leq&\frac{\gamma^{2k_0-1}}{(2k_0)^{2k_0-1}}\cdot\frac{\delta_2}{2k_0}\cdot\frac{2k_0T}{(1+(T-\delta_2)A)^2}\leq \frac{\gamma^{2k_0-1}}{(2k_0)^{2k_0-1}}\cdot\frac{\delta_2}{2k_0}\cdot\frac{2k_0T}{(1+AT)AT}\notag\\
\leq&
\frac{\gamma^{2k_0-1}}{(2k_0)^{2k_0-1}}\cdot\frac{\delta_2}{2k_0}\cdot1.5\log
k_0,
\end{align}
by recalling that $A\geq 0.69\log k_0$ and
$1+AT=2k_0\cdot(\log(2k_0))^{-2}$.

Third,
\begin{align}
&\int_{\Delta_{2k_0-1,1}}\prod_{j=1}^{2k_0-1}\frac{\1_{[1,T]}(2k_0t_j)}{(1+2k_0At_j)^2}\bigg(\int_0^{\delta_3/(2k_0)}\frac{4k_0^2t_{2k_0}d t_{2k_0}}{(1+2k_0At_{2k_0})^2}\bigg) d t_1\cdots d t_{2k_0-1}\notag\\
\leq&\frac{\gamma^{2k_0-1}}{(2k_0)^{2k_0-1}}\int_0^{\delta_3/(2k_0)}4k_0^2t_{2k_0}d t_{2k_0}=\frac{\gamma^{2k_0-1}}{(2k_0)^{2k_0-1}}\cdot\frac{\delta_3^2}{2}.
\end{align}
Finally, noting that
$$
\frac{t}{(1+2k_0At)^2}\leq\frac{1}{8k_0A}
$$ for $t\geq 0$
and letting $r=t_1+\cdots +t_{2k_0}$, we have
\begin{align*}
&\int_{\Delta_{2k_0,1}\setminus\Delta_{2k_0,1-\delta_1}}\frac{\1_{[1,T]}(2k_0t_{2k_0})t_{2k_0}}{(1+2k_0At_{2k_0})^2}\prod_{j=1}^{2k_0-1}\frac{\1_{[1,T]}(2k_0t_j)}{(1+2k_0At_j)^2}d t_1\cdots d t_{2k_0-1}d t_{2k_0}\\
\leq&\int_{\Delta_{2k_0-1,1}}\prod_{j=1}^{2k_0-1}\frac{\1_{[1,T]}(2k_0t_j)}{(1+2k_0At_j)^2}
\bigg(\int_{1-\delta_1}^1\frac{|r-t_1-\cdots-t_{2k_0-1}|d r}{(1+2k_0A|r-t_1-\cdots-t_{2k_0-1}|)^2}\bigg)d t_1\cdots d t_{2k_0-1}\\
\leq&\frac{\delta_1}{8k_0A}\cdot\frac{\gamma^{2k_0-1}}{(2k_0)^{2k_0-1}}\leq\frac{\delta_1}{5.5k_0\log k_0}\cdot\frac{\gamma^{2k_0-1}}{(2k_0)^{2k_0-1}}.
\end{align*}
Hence in view of (\ref{d1}) and (\ref{d2}), we obtain that
\begin{align}
\alpha\leq\bigg(\frac{3.01}{5.5\delta_1k_0\log k_0}+\frac{3.01\cdot 1.5\log k_0}{2\delta_2k_0}+\frac{3.01}{2}+\frac36\bigg)\cdot\frac{\gamma^{2k_0-1}}{(2k_0)^{2k_0-1}}\leq\frac{8.98\gamma^{2k_0-1}}{(2k_0)^{2k_0-1}}.
\end{align}

Similarly, we can get
$$
\beta_2=\int_{\Delta_{2k_0,1}}t_{2k_0}G_{2k_0}(t_1,\ldots,t_{2k_0})F(t_1,\ldots,t_{2k_0})d
t_1\cdots d t_{2k_0-1} d t_{2k_0} \leq
\frac{18}{k_0A}\cdot\frac{\gamma^{2k_0-1}}{(2k_0)^{2k_0-1}}.
$$

Thus applying Theorem \ref{tt}, for any $1\leq j\leq k_0$, we have
\begin{align}
&\sum_{\substack{x\leq n<2x\\ n\equiv b\pmod{W}\\ \mu(n+h_{2j})\neq0}}\theta(n+h_{2j-1})\tau(n+h_{2j})
\bigg(\sum_{d_j\mid n+h_j}\lambda_{d_1,\ldots,d_{2k_0}}\bigg)^2\notag\\
\leq&\log(2x)\sum_{\substack{x\leq n<2x\\ n\equiv b\pmod{W}}}\tau(n+h_{2j})
\bigg(\sum_{d_j\mid n+h_j}\lambda_{d_1,\ldots,d_{2k_0}}\bigg)^2\notag\\
\leq&\log(2x)\cdot\big(4.001\alpha+4\beta_2+o(1)\big)\cdot\frac{x}{(\log R)^{2k_0}}\cdot\frac{W^{2k_0-1}}{\phi(W)^{2k_0}}\notag\\
\leq& \frac{36.1\gamma^{2k_0-1}}{(2k_0)^{2k_0-1}}\cdot\frac{x\log x}{(\log R)^{2k_0}}\cdot\frac{W^{2k_0-1}}{\phi(W)^{2k_0}}.
\end{align}
Furthermore, combining (\ref{tsse}) and (\ref{sse}), we have
\begin{align}
&\sum_{\substack{x\leq n<2x\\ n\equiv b\pmod{W}\\ \mu(n+h_{2j})\neq0,\,1\leq j\leq k}}\bigg(\sum_{j=1}^{k_0}\theta(n+h_{2j-1})-m\log(3x)\bigg)
\bigg(\sum_{d_j\mid n+h_j}\lambda_{d_1,\ldots,d_{2k_0}}\bigg)^2
\notag\\
\geq&\bigg(\frac{8m+0.5}{2}-m\cdot\frac{\log
x}{\log
R}\bigg)\cdot\frac{\gamma^{2k_0}}{(2k_0)^{2k_0}}\cdot\frac{W^{2k_0-1}}{\phi(W)^{2k_0}}\cdot
\frac{(1+o(1))x}{(\log R)^{2k_0-1}}\notag\\
\geq&\frac{0.249\gamma^{2k_0}}{(2k_0)^{2k_0}}\cdot\frac{W^{2k_0-1}}{\phi(W)^{2k_0}}\cdot\frac{x}{(\log
R)^{2k_0-1}}.
\end{align}

It is easy to verify that
$$
\frac{\log k_0}{\log 2}-\frac{8m}{\log 2}-\frac{2\log m}{\log 2}\leq 11.6
\text{\qquad and\qquad}
-\frac{\log\gamma}{\log 2}-\frac{\log m}{\log 2}\leq 3.5.
$$
Thus letting
$$
C_2=1161k_0^2\gamma^{-1}
$$ in (\ref{C2}), we can get
\begin{equation}\label{C2log}
\Omega(n+h_{2j})\leq\frac{\log C_2}{\log2}=\frac{2\log k_0}{\log 2}-\frac{\log\gamma}{\log 2}+\frac{\log 1161}{\log 2}\leq\frac{16m}{\log 2}+\frac{5\log m}{\log 2}+36.9.
\end{equation}

Finally, let  $m=1$ so that $k_0=e^{16}\approx 8.9\times 10^6$ by (\ref{C1}). 
By the calculations of Nicely \cite{N} and Silva \cite{Si}, there exist more than $k_0$ twin-prime pairs in the interval $[2\times 10^7,3\times 10^9]$.
Let $(h_{2j-1},h_{2j})$, $j=1,\ldots,k_0$ be distinct twin-prime pairs in $[2\times 10^7,3\times 10^9]$.
Clearly $\{h_1,\ldots,h_{2k_0}\}$ is admissible. Thus in view of (\ref{C2log}), 
for arbitrarily large $x$, there exist $n\in[x,2x]$ and $1\leq j_1<j_2\leq k_0$ such that $n+h_{2j_1-1}$, $n+h_{2j_2-1}$ are primes and
$$
\Omega(n+h_{2j_1}),\  \Omega(n+h_{2j_2})\leq
\frac{16}{\log 2}+36.9<60,
$$
\qed
\begin{remark}
Suppose that $k_0\geq m^2e^{4m+8}$ and $\{h_1,\ldots,h_{k_0}\}$ is admissible. Then we also have
$$
\sum_{\substack{x\leq n<2x\\ n\equiv b\pmod{W}\\ \mu(n+h_{j})\neq0,\,1\leq j\leq k}}\bigg(\sum_{j=1}^{k_0}\theta(n+h_{j})\bigg(1-\sum_{i\neq j}\frac{\tau(n+h_{i})}{C_3}\bigg)-m\log(3x)\bigg)
\bigg(\sum_{d_j\mid n+h_j}\lambda_{d_1,\ldots,d_{k_0}}\bigg)^2>0,
$$
where $C_3=512k_0^3\log k_0$. Hence, there exist infinitely many $n$ such that
$$
\{n+h_1,n+h_2,\cdots,n+h_{k_0}\}
$$
contains at least $m+1$ primes and
$$
\Omega(n+h_j)\leq\frac{3\log k_0}{\log 2}+\frac{\log\log k_0}{\log 2}+9
$$
for each $j$.
\end{remark}

\section*{Acknowledgement}
\medskip

We are grateful to the annoymous referee for his/her very helpful comments and suggestions on this paper.
This work is supported by the National Natural Science Foundation of
China (Grant No. 11271249) and the Specialized Research Fund for the
Doctoral Program of Higher Education (No. 20120073110059).

\end{document}